\documentclass{amsart}

\usepackage{latexsym,amssymb,amsthm}
\usepackage{hyperref}

\usepackage{cleveref}
\renewcommand{\wr}{\mathop{\mathrm{wr}}}
\newcommand{\Sym}{\mathop{\mathrm{Sym}}}
\newcommand{\Aut}{\mathop{\mathrm{Aut}}}

\newtheorem{theorem}{Theorem}[section]
\newtheorem{lemma}[theorem]{Lemma}
\newtheorem{example}[theorem]{Example}
\newtheorem{proposition}[theorem]{Proposition}
\newtheorem{remark}[theorem]{Remark}
\newtheorem{definition}[theorem]{Definition}
\def\Zent#1{{\bf Z}({{{#1}}})}
\def\Zentm#1#2{{\bf Z}_{{#1}}({{#2}})}
\def\cent#1#2{{\bf C}_{{#1}}({{{#2}}})}

\begin{document}
\title[Transitive nilpotent and soluble groups]{On the maximum order of nilpotent transitive  permutation groups}

\author[E. Crestani]{Eleonora Crestani}
\address{Eleonora Crestani, Dipartimento di Matematica, University of Padova, Via Trieste 63, 35131 Padova, Italy} \email{crestani@math.unipd.it}

\author[P. Spiga]{Pablo Spiga}
\address{Pablo Spiga, Dipartimento di Matematica e Applicazioni, University of Milano-Biccoca, Via Cozzi 55, 20125 Milano, Italy} \email{pablo.spiga@unimib.it}

\thanks{Address correspondence to P. Spiga,
E-mail: pablo.spiga@unimib.it}


\begin{abstract}
Given two positive integers $n$ and $c$, we determine an upper bound, as a function of $n$ and $c$, for the maximum order of a finite nilpotent transitive  group of degree $n$ and nilpotency class at most $c$. 
\end{abstract}
\maketitle

\section{Introduction}\label{intro}

It is well-known and easy to show that every transitive abelian group of degree $n$ acts regularly, and hence has order equal to $n$. This behaviour is peculiar to abelian groups, but one might wonder, if minded so,  whether there is any  relationship among the order $|G|$,  the degree $n$ and the nilpotency class $c$ of a nilpotent transitive  permutation group $G$. Broadly speaking, in this paper we uncover this relationship.

Clearly, when $G$ is non-abelian, $|G|$ is not uniquely determined by  $n$ and $c$: for instance, the symmetric group $\Sym(16)$ contains nilpotent transitive groups of class $3$ and order $2^\ell$, with $\ell$ ranging from $4$ to $10$, see~\cite{magma}.  Therefore, it is natural to look for lower and upper bounds for $|G|$, as functions of $n$ and $c$. Actually, in this paper, we are only concerned with upper bounds for $|G|$ (as functions of $n$ and $c$). We have two reasons to make this choice. First, our motivation for this work is spurred by the proofs of~\cite[Theorem 2]{PSV} 
and~\cite[Theorem $1.3$]{SV}. In these papers the authors are interested in bounding the order of the vertex-stabilisers of certain finite vertex-transitive graphs, and at a critical juncture they need to estimate the maximum order of a transitive $p$-group of degree $n$ and nilpotency class at most $3$. In both papers this problem was by-passed using the rigid structure of the Sylow subgroups of the vertex-stabilisers  under consideration, however this resulted in very technical arguments and (in our opinion) in an artificial detour. In particular, our work (specifically Theorem~\ref{thrm:main} below) can be used to simplify the proofs of~\cite[Theorem~$2$]{PSV} and~\cite[Theorem~$1.3$]{SV}. Second, the function determining the minimum order of a  nilpotent transitive group of degree $n$ and nilpotency class $c$ seems to be rather irregular and dramatically influenced by the prime factorisation of $n$. Moreover, when $n=p^k$ for some prime number $p$ and some $k\geq 1$, Remark~\ref{rem10} suggests that a lower bound for $|G|$ becomes an interesting and meaningful problem only when $c$ is large compared to $k$.

\begin{definition}\label{def1} {\rm Let $n$ and $c$ be positive integers. We define
\begin{multline*}
F_\mathrm{Nil}(n,c):=\max\{|G|\mid G \textrm{ nilpotent transitive permutation group}\\ 
\textrm{of degree }n \textrm{ and nilpotency class at most }c \}.
\end{multline*}
Clearly, $F_\mathrm{Nil}(n,1)=n$.
}
\end{definition}
Our choice of  allowing nilpotent transitive permutation groups of nilpotency class at most $c$ (rather than equal to $c$) in Definition~\ref{def1} is aesthetic: it makes the statements of our main results clearer and, in our view, more natural. 

Our first result reduces the problem of computing $F_\mathrm{Nil}(n,c)$ to the analogous problem on prime power degrees. 
\begin{theorem}\label{thrm:main-1}
Let $n$ and $c$ be positive integers and let $n=p_1^{\alpha_1}\cdots p_\ell^{\alpha_\ell}$ be the prime factorisation of $n$, with $p_1,\ldots,p_\ell$ distinct prime numbers and $\alpha_1,\ldots,\alpha_\ell\geq 1$. Then
$$F_\mathrm{Nil}(n,c)=\prod_{i=1}^{\ell} F_\mathrm{Nil}(p_i^{\alpha_i},c).$$
\end{theorem}
In view of Theorem~\ref{thrm:main-1}, we need only to consider the case when $n$ is a prime power. Our second result gives an explicit formula for $F_\mathrm{Nil}(p^k,2)$ and a description of the nilpotent transitive groups $G$ of degree $p^k$, nilpotency class at most $2$ and  order  $F_\mathrm{Nil}(p^k,2)$.
\begin{theorem}\label{thrm:main-2}
Let $p$ be a prime number and let $k\geq 1$. We have $$\log_p(F_\mathrm{Nil}(p^k,2))= k+\left\lfloor\frac{k}{2}\right\rfloor\left\lceil\frac{k}{2}\right\rceil.$$
Moreover, if $G$ is a nilpotent transitive group of degree $p^k$, nilpotency class at most $2$ and order $F_\mathrm{Nil}(p^k,2)$, then $G$ is permutation isomorphic to one of the groups  in Example~$\ref{ex2}$.
\end{theorem}
In this opening section we do not introduce the groups in Example~\ref{ex2} because a precise definition would take us too far astray. We refer also  to Example~\ref{ex1} and Remarks~\ref{rem1} and~\ref{rem2} for more information on the groups in Example~\ref{ex2}.

\begin{definition}\label{defdefdef}
{\rm Let $k$ and $c$ be positive integers. We let
$$
F(k,c):=\max\left(\sum_{i=1}^{c}a_i\left(\frac{\left(\sum_{j=1}^{i-1}a_j\right)^i-1}{\left(\sum_{j=1}^{i-1}a_j\right)-1}\right) \mid a_1,\ldots,a_c\in \mathbb{N}, k=\sum_{j=1}^{c} a_{j}\right),
$$
With the help of MAGMA, we have tabulated $F(k,c)$ in Table~\ref{tableb1}, for $c\leq 4$. 

In Proposition~\ref{leadcoeff}, we show that, when $c$ is fixed, $$F(k,c)=\frac{(c-1)^{c-1}k^c}{c^c}+o(k^{c-1}).$$ Hence, for $c$ fixed, $F(k,c)$ can be thought of as a polynomial in $\mathbb{Q}[k]$ of degree $c$ and having leading coefficient $$\frac{(c-1)^{c-1}}{c^c}=\frac{1}{c-1}\left(1-\frac{1}{c}\right)^{c}.$$In particular,  the leading coefficient of $(c-1)F(k,c)$ is asymptotic to the ubiquitous $1/e$.
}
\end{definition}
\begin{theorem}\label{thrm:main}Let $p$ be a prime number, let $k\geq 1$ and let $F(k,c)$ be as in Definition~$\ref{defdefdef}$. Then $\log_p(F_\mathrm{Nil}(p^k,c))\leq F(k,c)$.
\end{theorem}

In Proposition~\ref{rem3}, we prove that the upper bound in Theorem~\ref{thrm:main} is (for $c$ small compared to $k$) not far from best possible. Specifically, we show that $\log_p(F_\mathrm{Nil}(p^k,c))$ is bounded from below by a polynomial in $k$ of degree $c$.

We conclude this introductory section referring to the beautiful article~\cite{Palfy}. Here, P\'alfy surveys the known estimates of the form ``$|G|\leq f(n)$'' for (not necessarily transitive) permutation groups $G$ of degree $n$, where the function $f$ depends on whether the group $G$ is abelian, nilpotent or soluble. See~\cite[Table~$1$]{Palfy} for a comprehensive table of classical results, mostly taken from~\cite{Dixon}.  The investigation of P\'alfy differs from ours in at least two important facts. First, we are only concerned with transitive groups.  Second (and more importantly), our estimates are functions of both the degree and the nilpotency class. For instance, from~\cite[Table~$1$]{Palfy} we infer that a transitive nilpotent group $G$ of degree $n$ has order at most $2^{n-1}$, which is exponential in $n$. Theorem~\ref{thrm:main} reveals that, once the nilpotency class of $G$ is fixed, $|G|$ is poly-logarithmic in $n$. 
\begin{table}[!ht]
\begin{tabular}{|c|c|l|}\hline
$c$&$k$&$F(k,c)$\\\hline
$1$&any&$k$\\
$2$&any&$\lfloor \frac{k}{2}\rfloor\lceil \frac{k}{2}\rceil+k$\\
$3$&$k\equiv 0\mod 3$&$\frac{4}{27}k^3+\frac{1}{3}k^2+k$\\
$3$&$k\equiv 1\mod 3$&$\frac{4}{27}k^3+\frac{1}{3}k^2+\frac{8}{9}k-\frac{10}{27}$\\
$3$&$k\equiv 2\mod 3$&$\frac{4}{27}k^3+\frac{1}{3}k^2+\frac{8}{9}k-\frac{8}{27}$\\
$4$&2&$5$\\
$4$&$6$&$188$\\
$4$&$k\equiv 0\mod 4$&$\frac{27}{256}k^4+\frac{13}{64}k^3+\frac{3}{8}k^2+k$\\
$4$&$k\equiv 1\mod 4$&$\frac{27}{256}k^4+\frac{13}{64}k^3+\frac{41}{128}k^2+\frac{53}{64}k-\frac{117}{256}$\\
$4$&$k\equiv 2\mod 4$, $k>6$&$\frac{27}{256}k^4+\frac{13}{64}k^3+\frac{1}{8}k^2+\frac{7}{16}k-\frac{11}{16}$\\
$4$&$k\equiv 3\mod 4$&$\frac{27}{256}k^4+\frac{13}{64}k^3+\frac{37}{128}k^2+\frac{57}{64}k-\frac{77}{256}$\\
\hline
\end{tabular}
\caption{$F(k,c)$, for $c\leq 4$}\label{tableb1}
\end{table}

\subsection{Notation}\label{sub1}Our notation is standard. Given a finite nilpotent group $G$, we denote by $\gamma_i(G)$ the $i^{\mathrm{th}}$-term of the lower central series of $G$, and by abuse of notation, we write $\gamma_1(G)=G$. As usual $\Zent G$ denotes the centre of $G$ and $\Zentm i G$ denotes the $i^{\mathrm{th}}$-term of the upper central series of $G$. Given two subsets $X$ and $Y$ of $G$, the commutator subgroup of $X$ and $Y$ is denoted by $[X,Y]:=\langle [x,y]\mid x\in X,y\in Y\rangle$, and the centraliser of $X$ in $G$ by $\cent G X$.

If $G$ is a permutation group on $\Omega$ and $\omega\in \Omega$, then we denote by $G_\omega$ the stabiliser in $G$ of the point $\omega$.

\subsection{Structure of the paper}\label{Sub2}In Section~\ref{preliminaries} we prove Theorem~\ref{thrm:main-1} (which is an easy observation). We prove Theorem~\ref{thrm:main-2} in Section~\ref{class2} and Theorem~\ref{thrm:main} in Section~\ref{main}.

\subsection{Acknowledgements}\label{Sub3}The ideas that gave rise to this work owe very much to  some conversations that the second author had with Gabriel Verret during his visit at the University of Western Australia in $2014$. He thanks UWA for the warm hospitality, Gab for the constant encouragement and Gordon Royle for sharing his wisdom on caramelised onion on lamb chops.
 
\section{Preliminaries}\label{preliminaries}
Let $H$ and $K$ be transitive groups on $\Delta$ and $\Lambda$, respectively.  The direct product $H\times K$ acts transitively on the set $\Delta\times \Lambda$ by setting $(\delta,\lambda)^{(h,k)}=(\delta^h,\lambda^k)$, for each $(\delta,\lambda)\in \Delta\times \Lambda$ and each $(h,k)\in H\times K$. We refer to this faithful action of $H\times K$ as the \textit{natural product action}.

Let $n$ be a positive integer and let $n=p_1^{\alpha_1}\cdots p_\ell^{\alpha_\ell}$ be the prime factorisation of $n$, with $p_1,\ldots,p_\ell$ prime numbers and $\alpha_1,\ldots,\alpha_\ell\geq 1$. Let $G$ be a nilpotent transitive group  on the set $\Omega$ of cardinality $n$ and  let $\omega\in \Omega$. Let $R$ be a Sylow $r$-subgroup of $G$ with $r\nmid n$. As $|G:G_\omega|=n$ is coprime to $r$, by Sylow's theorems $R$ is conjugate to a subgroup of  $G_\omega$. Since $G$ is nilpotent, we have $R\unlhd G$ and, since $G_\omega$ is core-free in $G$, we have $R=1$. This shows that $G=P_1\times \cdots \times P_\ell$, where $P_i$ is the Sylow $p_i$-subgroup of $G$ for each $i\in \{1,\ldots,\ell\}$.  Write $G_\omega=Q_1\times \cdots \times Q_\ell$, where $Q_i:=(P_i)_\omega$ is the Sylow $p_i$-subgroup of $G_\omega$. We let $\Omega_i$ denote the $P_i$-orbit containing $\omega$, that is, $\Omega_i:=\omega^{P_i}=\{\omega^g\mid g\in P_i\}$.

\begin{proposition}\label{prop1}Assume the notation we have established above. The action of $G$ on $\Omega$ is permutation isomorphic to the natural product action of $G=P_1\times \cdots \times P_\ell$ on $\Omega_1\times \cdots \times \Omega_\ell$. 
\end{proposition}
\begin{proof}
Observe that $|\Omega_i|=|P_i:(P_i)_\omega|=|P_i:Q_i|=p_i^{\alpha_i}$ and hence $|\Omega|=\prod_{i=1}^\ell |\Omega_i|$. Since $P_1,\ldots,P_\ell$ are pair-wise commuting subgroups of $G$, the mapping $$\iota:\Omega_1\times \cdots \times\Omega_\ell\to \Omega$$ defined by $(\omega^{g_1},\ldots,\omega^{g_\ell})\mapsto \omega^{g_1\cdots g_\ell}$ (with $g_i\in P_i$, for each $i$) is well-defined and injective. Thus $\iota$ is bijective. Now, $\iota$ determines a permutation isomorphism between  the natural product action of $G$ on $\Omega_1\times\cdots \times \Omega_\ell$ and the action of $G$ on $\Omega$.
\end{proof}

\begin{proof}[Proof of Theorem~$\ref{thrm:main-1}$]
For each $i\in \{1,\ldots,\ell\}$, let $P_i$ be a nilpotent transitive group of degree $p_i^{\alpha_i}$ and nilpotency class at most $c$ with $|P_i|=F_\mathrm{Nil}(p_i^{\alpha_i},c)$. Then $G:=P_1\times\cdots\times P_\ell$, endowed with its natural product action, is transitive of degree $\prod_{i=1}^\ell p_i^{\alpha_i}=n$, has nilpotency class at most $c$, and $|G|=\prod_{i=1}^\ell F_\mathrm{Nil}(p_i^{\alpha_i},c)$. Thus $\prod_{i=1}^\ell F_\mathrm{Nil}(p_i^{\alpha_i},c)\leq F_\mathrm{Nil}(n,c)$. 

Conversely, let $G$ be a nilpotent transitive group of degree $n$ and nilpotency class at most $c$ with $|G|=F_\mathrm{Nil}(n,c)$. By Proposition~\ref{prop1}, $G=P_1\times \cdots \times P_\ell$,  $P_i$ is the Sylow $p_i$-subgroup of $G$, and $P_i$ has a faithful transitive action of degree $p_i^{\alpha_i}$. Thus $F_\mathrm{Nil}(n,c)=|G|=\prod_{i=1}^\ell|P_i|\leq \prod_{i=1}^\ell F_\mathrm{Nil}(p_i^{\alpha_i},c)$.
\end{proof}

The following two facts hardly deserve to be called lemmas, but will be used several times. 
\begin{lemma}\label{retard}
Let $K$ be a transitive permutation group on $\Omega$. If $K$ centralises $g\in \Sym(\Omega)$ and $g$ fixes some point of $\Omega$, then $g=1$.
\end{lemma}
\begin{proof}
Let $\omega$ be an element of $\Omega$ fixed by $g$. For every $k\in K$, we have $(\omega^{k})^g=\omega^{kg}=\omega^{gk}=\omega^k$ and hence $g$ fixes $\omega^k$. Since $K$ is transitive,  $g$ fixes every element of $\Omega$ and hence $g=1$.
\end{proof}

\begin{lemma}\label{retardd}
Let $k$ and $\ell$ be  positive integers, let $K$ be a  group and let $H_1,\ldots,H_\ell$ be subgroups of $K$ with $|K:H_i|=k$, for each $i$. Then $|K:\cap_{i=1}^\ell H_i|\leq k^\ell$.
\end{lemma}
\begin{proof}
Note that if $A$ and $B$ are subgroups of $K$, then $|K:A\cap B|\leq |K:A||K:B|$. Now, the proof follows by induction on $\ell$.
\end{proof}

Now we prove an elementary upper bound for $F_\mathrm{Nil}(p^k,c)$. This upper bound is weaker then the upper bound in Theorem~\ref{thrm:main}. However, its proof is elementary and contains the main ingredients (but not the  technicalities) of the proof of Theorem~\ref{thrm:main}. In particular, we hope  this makes the flow of the argument in Theorem~\ref{thrm:main} easier to follow.
\begin{proposition}\label{propo1}Let $k$ and $c$ be positive integers and let $p$ be a prime number. Then $\log_p(F_\mathrm{Nil}(p^k,c))\leq k(k^c-1)/(k-1)$.
\end{proposition}
\begin{proof}
Let $G$ be a nilpotent transitive permutation group of degree $p^k$ and nilpotency class at most $c$, let $\Omega$ be the set acted upon by $G$, let $\omega\in \Omega$ and let $H$ be the point stabiliser $G_\omega$. Observe that $\gamma_c(G)\leq \Zent G$ and hence $\gamma_c(G)\cap H=1$ by Lemma~\ref{retard}. We show that $\log_p(|G|)\leq k(k^c-1)/(k-1)$, from which the proof follows from the definition of $F_\mathrm{Nil}(p^k,c)$.

Set $g_0:=1$ and $K_0:=\langle g_0\rangle$. For $i>0$, we define recursively $g_i$ and $K_i$. If $K_{i-1}H$ is a proper subset of $G$, then choose $g_i\in G\setminus K_{i-1}H$ and set $K_i:=\langle K_{i-1},g_i\rangle$. If $G=K_{i-1}H$, then choose $g_i:=1$ and $K_i:=\langle K_{i-1},g_i\rangle=K_{i-1}$. Let $\ell$ be the minimum positive integer with $K_{\ell}=K_{\ell+1}$. By construction, $G=K_\ell H$ and hence $K_\ell$ is transitive on $\Omega$. Moreover, since $|G:H|=|\Omega|=p^k$ and $|K_iH|<|K_{i+1}H|$ for each $i\in \{0,\ldots,\ell-1\}$, we have $\ell\leq k$.

For $\kappa\in \{1,\ldots,c-1\}$, define $$L_{\kappa}:=\bigcap_{i_1,\ldots,i_{c-\kappa}=0}^\ell (\gamma_{\kappa}(G)\cap H)^{g_{i_1}\cdots g_{i_{c-\kappa}}}.$$
By taking $i_1=\cdots =i_{c-\kappa}=0$ we see that $\gamma_{\kappa}(G)\cap H$ is one of the terms of this intersection, and hence $L_\kappa\leq \gamma_\kappa(G)\cap H$. 

We claim that $L_\kappa=1$. We argue inductively on $c-\kappa$.  If $c-\kappa=1$, then $\kappa=c-1$ and $L_{c-1}=\cap_{i=0}^{\ell}(\gamma_{c-1}(G)\cap H)^{g_{i}}$. Let $y$ be in $L_{c-1}$. As $y\in (\gamma_{c-1}(G)\cap H)^{g_{i}}$, there exists $h_i\in \gamma_{c-1}(G)\cap H$ with $y=h_i^{g_{i}}$. Since $h_i$ and $y$ are both in $H$, we have $h_i^{-1}y=[h_i,g_{i}]\in \gamma_c(G)\cap H=1$, and hence $h_i=y$. It follows that $y=h_i^{g_{i}}=y^{g_{i}}$ and hence $y$ is centralised by $g_{i}$. Therefore $y$ is centralised by $\langle g_{0},\ldots,g_{\ell}\rangle=K_\ell$. Since $K_\ell$ is transitive, by Lemma~\ref{retard}, we obtain $y=1$, and the base case of the induction is proved.  

Now suppose that $c-\kappa>1$. Let $y$ be in $L_{\kappa}$. Fix $i_1,\ldots,i_{c-\kappa}\in \{0,\ldots,\ell\}$. As $y\in (\gamma_\kappa(G)\cap H)^{g_{i_1}\cdots g_{i_{c-\kappa}}}$, there exists $h_{i_1,\ldots,i_{c-\kappa}}\in \gamma_\kappa(G)\cap H$ with $$y=(h_{i_1,\ldots,i_{c-\kappa}})^{g_{i_1}\cdots g_{i_{c-\kappa}}}.$$
Thus $(h_{i_1,\ldots,i_{c-\kappa}})^{-1}y=[h_{i_1,\ldots,i_{c-\kappa}},g_{i_1}\cdots g_{i_{c-\kappa}}]$. As both $y$ and $h_{i_1,\ldots,i_{c-\kappa}}$ are in $H$, we may write $h_{i_1,\ldots,i_{c-\kappa}}=yz_{i_1,\ldots,i_{c-\kappa}}$, where 
$z_{i_1,\ldots,i_{c-\kappa}}=[h_{i_1,\ldots,i_{c-\kappa}},g_{i_1}\cdots g_{i_{c-\kappa}}]^{-1}\in \gamma_{\kappa+1}(G)\cap H$. We obtain 
\begin{equation}\label{Eeqnew1}
y=(yz_{i_1,\ldots,i_{c-\kappa}})^{g_{i_1}\cdots g_{i_{c-\kappa}}}.
\end{equation}

Applying Eq.~\eqref{Eeqnew1} with $i_{c-\kappa}=0$ and recalling that $g_{0}=1$, we have 
\begin{equation}\label{Eeqnew2}
y^{g_{i_1}\cdots g_{i_{c-(\kappa+1)}}}=y((z_{i_1,\ldots,i_{c-(\kappa+1),0}})^{-1})^{g_{i_1}\cdots g_{i_{c-(\kappa+1)}}}.
\end{equation}
Eq.~\eqref{Eeqnew1} (applied twice, first to the $(c-\kappa)$-tuple $0,\ldots,0,i_{c-\kappa}$ and then to the $(c-\kappa)$-tuple $i_1,\ldots,i_{c-\kappa}$) gives
 \begin{eqnarray*}
(yz_{0,\ldots,0,i_{c-\kappa}})^{g_{i_{c-\kappa}}}&=&y=(yz_{i_1,\ldots,i_{c-\kappa}})^{g_{i_1}\cdots g_{i_{c-\kappa}}}\\
&=&\left( (yz_{i_1,\ldots,i_{c-\kappa}})^{g_{i_1}\cdots g_{i_{c-(\kappa+1)}}}\right)^{g_{i_{c-\kappa}}}
\end{eqnarray*}
and hence
\begin{equation*}
yz_{0,\ldots,0,i_{c-\kappa}}=(yz_{i_1,\ldots,i_{c-\kappa}})^{g_{i_1}\cdots g_{i_{c-(\kappa+1)}}}=y^{g_{i_1}\cdots g_{i_{c-(\kappa+1)}}}(z_{i_1,\ldots,i_{c-\kappa}})^{g_{i_1}\cdots g_{i_{c-(\kappa+1)}}}.
\end{equation*}
From this and Eq.~\eqref{Eeqnew2}, we obtain
\begin{equation}\label{Eeqnew3}
z_{0,\ldots,0,i_{c-\kappa}}=((z_{i_1,\ldots,i_{c-(\kappa+1),0}})^{-1}z_{i_1,\ldots,i_{c-\kappa}})^{g_{i_1}\cdots g_{i_{c-(\kappa+1)}}}.
\end{equation}
In particular, $z_{0,\ldots,0,i_{c-\kappa}}\in (\gamma_{\kappa+1}(G)\cap H)^{g_{i_1}g_{i_2}\cdots g_{i_{c-(\kappa+1)}}}$.
As $i_{1},\ldots,i_{c-\kappa}$ is an arbitrary $(c-\kappa)$-tuple of elements of $\{0,\ldots,\ell\}$, Eq.~\eqref{Eeqnew3} gives
$$z_{0,\ldots,0,i_{c-\kappa}}\in \bigcap_{i_1,\ldots,i_{c-(\kappa+1)}=0}^{\ell}(\gamma_{\kappa+1}(G)\cap H)^{g_{i_1}g_{i_2}\cdots g_{i_{c-(\kappa+1)}}}=L_{\kappa+1}=1,$$
where in the last equality we used the inductive hypothesis. Thus $z_{0,\ldots,0,i_{c-\kappa}}=1$.

Applying Eq.~\eqref{Eeqnew1} with $i_1=\cdots=i_{c-(\kappa+1)}=0$, we deduce $y=y^{g_{i_{c-\kappa}}}$. Therefore $y$ is centralised by $\langle g_0,\ldots,g_\ell\rangle=K_\ell$. Now Lemma~\ref{retard} yields $y=1$. As $y$ is an arbitrary element of $L_\kappa$, we get $L_\kappa=1$, and this concludes the proof of our claim.

For $\kappa=1$, we have 
\begin{equation}\label{Eeq1}
L_1=\bigcap_{i_1,\ldots,i_{c-1}=0}^\ell H^{g_{i_1}\cdots g_{i_{c-1}}}=1.
\end{equation}
Observe that in the intersection in~\eqref{Eeq1} there are exactly $1+\ell+\cdots+\ell^{c-1}=(\ell^c-1)/(\ell-1)$  terms: $1$ accounts for the term of the intersection with $i_1=\cdots=i_\kappa=0$, $\ell$ accounts for the terms of the intersection with exactly one $i_j\neq 0$, etc. In particular, $L_1$ is the intersection of $(\ell^c-1)/(\ell-1)$ subgroups of $G$  having index $p^k$. Therefore, by Lemma~\ref{retardd}, $$|G|=|G:L_1|\leq \left(p^{k}\right)^{(\ell^c-1)/(\ell-1)}.$$
As $\ell\leq k$, the proposition is proved.
\end{proof}
When $c=1$, Proposition~\ref{propo1} gives $F_\mathrm{Nil}(p^k,1)\leq p^k$, and hence  this upper bound equals $F_{\mathrm{Nil}}(p^k,1)$. However, for larger values of $c$, the upper bound in Proposition~\ref{propo1} is far from optimal. For instance, when $c=2$, Proposition~\ref{propo1} gives $F_\mathrm{Nil}(p^k,2)\leq p^{k+k^2}$, but in fact Theorem~\ref{thrm:main-2} shows that $F_\mathrm{Nil}(p^k,2)=p^{k+\lfloor k/2\rfloor\lceil k/2\rceil}$.

We conclude this preliminary section making an observation on the minimal order $f_\mathrm{Nil}(p^k,c)$ of a nilpotent transitive group of degree $p^k$ and nilpotency class $c$.
\begin{remark}\label{rem10}{\rm Suppose that $c\leq k-1$. Let $P$ be a group of order $p^{c+1}$ and maximal class, that is, $P$ has nilpotency class $c$. Let $Q$ be an abelian group of order $p^{k-c-1}$ and set $G:=P\times Q$. Then $G$ has nilpotency class $c$ and, via its right regular representation, is transitive of degree $p^k$. Therefore $f_\mathrm{Nil}(p^k,c)=p^k$. This suggests that the function $f_\mathrm{Nil}(p^k,c)$  becomes interesting only when $c$ is large compared to $k$. 

We make another example along these lines. This example is inspired by~\cite[Section~$3$, Exercise~$3.4$]{Gu} (see also~\cite[Proposition~$9.15$]{Berkovich}). Suppose that $c\leq (k-1)(p-1)$ and write $c=q(p-1)+r$ with $0\leq r<p-1$. Observe that $q<k-1$ if $r>0$, and $q\leq k-1$ if $r=0$.  From~\cite[Exercise~$3.4$]{Gu}, there exists a $p$-group $P$ of maximal class $c$ with $P=M\rtimes \langle x\rangle$ such that
\begin{itemize}
\item[(i)] $x$ has order $p$,
\item[(ii)] $M=\langle x_1\rangle\times \langle x_2\rangle\times\cdots\times \langle x_{p-1}\rangle$ is abelian,
\item[(iii)]  $|x_1|=\cdots =|x_r|=p^{q+1}$ and $|x_{r+1}|=\cdots =|x_{p-1}|=p^q$, 
\item[(iv)]  $\Zent P=\langle x_r^{p^q}\rangle$ if $r>0$, and $\Zent P=\langle x_{p-1}^{p^{q-1}}\rangle$ if $r=0$.
\end{itemize}

Set $H:=\langle x_1,\ldots,x_{r-1},x_{r+1},\ldots,x_{p-1}\rangle$ if $r>0$, and $H:=\langle x_1,\ldots,x_{p-2}\rangle$ if $r=0$. Now, $H\cap \Zent P=1$ and hence $H$ is core-free in $P$. Let $\Delta$ be the set of right cosets of $H$ in $P$. Clearly, $|\Delta|=|P:H|=p^{q+2}$ if $r>0$, and $|\Delta|=p^{q+1}$ if $r=0$. Let $Q$ be an abelian group of order $p^{k-q-2}$ if $r>0$, and order $p^{k-q-1}$ if $r=0$. Set $G:=P\times Q$ and $\Omega:=\Delta\times Q$. Then $G$ has nilpotency class $c$ and, via its natural product action,  is transitive and faithful on $\Omega$ of degree $p^k$. Now, $|G|=p^{k+q(p-2)+r-1}$ if $r>0$, and $|G|=p^{k+q(p-2)}$ if $r=0$. As $q=(c-r)/(p-1)$, in both cases we obtain $|G|\leq p^{k+c(p-2)/(p-1)}$. Therefore $$p^k\leq f_\mathrm{Nil}(p^k,c)\leq p^{k+c(p-2)/(p-1)}< p^{k+c}.$$

}

\end{remark}

\section{Nilpotency class $2$}\label{class2}

Before moving to transitive $p$-groups of arbitrary nilpotency class $c$, we start a detailed discussion of the transitive $p$-groups of nilpotency class $2$. First, this has the benefit of highlighting the type of improvements that are needed in Proposition~\ref{propo1} to obtain Theorem~\ref{thrm:main}. Second, when $c=2$, we give an explicit formula for $F_\mathrm{Nil}(p^k,2)$ and we describe the transitive $p$-groups $G$ of nilpotency class $2$ with $|G|=F_\mathrm{Nil}(p^k,2)$, which gives a complete understanding of the case $c\leq 2$.

We start with an elementary example, which will hopefully help to follow the more general construction in Example~\ref{ex2} and the proof of Theorem~\ref{thrm:main-2}.
\begin{example}\label{ex1}{\rm
Let $p$ be a prime, let $k$ be a positive integer and let $m\in \{\lfloor k/2\rfloor,\lceil k/2\rceil\}$.  Let $V=\langle e_1,\ldots,e_k\rangle$ be the $k$-dimensional vector space of row vectors over the field $\mathbb{F}_p$ of size $p$, and let  $H$ be the group of  (block) unitriangular matrices
\[
H=\left\{
\left(
\begin{array}{cc}
I_{m}&0\\
A&I_{k-m}\\
\end{array}
\right)\mid A \textrm{ is an }(k-m)\times m\textrm{-matrix with coefficients in }\mathbb{F}_p
\right\}\]
where, for a positive integer $\ell$, $I_\ell$ denotes the $\ell\times \ell$-identity matrix with coefficients in $\mathbb{F}_p$. Let $G$ be the affine group of linear transformations $H\ltimes V$. The group $G$ acts faithfully and transitively on the underlying set $V$, with $V$ acting by translations and with $H$ acting by matrix multiplication, that is, $v^{(h,w)}:=vh+w$, for each $v\in V$ and $(h,w)\in H\ltimes V$.

A computation shows that, for $k>1$, $\Zent G=\langle e_1,\ldots,e_m\rangle=\gamma_2(G)$ and hence $G$ has nilpotency class $2$. Clearly $\log_p(|G|)=\log_p(|H||V|)={k+m(k-m)}={k+\lfloor k/2\rfloor\lceil k/2\rceil}$. }
\end{example}

In the next example we generalise the construction in Example~\ref{ex1}.
\begin{example}\label{ex2}
{\rm 
Let $p$ be a prime, let $k$ be a positive integer, let $m\in \{\lfloor k/2\rfloor,\lceil k/2\rceil\}$,  let $V$ be an abelian group of order $p^k$ and let $Z$ be a subgroup of $V$. Suppose that $|Z|=p^m$ and that both $V/Z$ and $Z$ are elementary abelian. Let $v_1,\ldots,v_{k-m}$ be generators of $V$ modulo $Z$, let $W:=\langle v_1^p,\ldots,v_{k-m}^p\rangle$ and let $Z'$ be a complement of $W$ in $Z$. Write $p^a:=|W|$, for some $a\in \{0,\ldots,m\}$. Let $w_1,\ldots,w_{m-a}$ be generators of $Z'$. 

Since the Frattini subgroup of $V$ is $W$, we see that $V$ is generated by $$S:=\{v_1,\ldots,v_{k-m},w_1,\ldots,w_{m-a}\}.$$ For every $z_1,\ldots,z_{k-m}\in Z$, let $\varphi_{z_1,\ldots,z_{k-m}}$ be the automorphism of $V$ defined on the generating set $S$ by
\begin{align*}
&&v_i^{\varphi_{z_1,\ldots,z_{k-m}}}&:=v_iz_i,&\textrm{for every }i\in \{1,\ldots,k-m\},\\
&&w_j^{\varphi_{z_1,\ldots,z_{k-m}}}&:=w_j,&\textrm{for every }j\in \{1,\ldots,m-a\},
\end{align*}
and let $H$ be the subgroup of $\Aut(V)$ generated by all such functions. Clearly, $H$ is an elementary abelian $p$-group of order $|Z|^{k-m}=p^{m(k-m)}$.

 Let $G$ be the semidirect product $V\rtimes H$. As in Example~\ref{ex1}, the group $G$ acts faithfully and transitively on the underlying set $V$, with $V$ acting by right multiplications and with $H$ acting by conjugation.

A computation shows that, for $k>1$, $\Zent G=Z=\gamma_2(G)$ and hence $G$ has nilpotency class $2$. Clearly $\log_p(|G|)={k+\lfloor k/2\rfloor\lceil k/2\rceil}$. 
}
\end{example}

\begin{remark}\label{rem1}{\rm A tedious computation and the structure theorem of finite abelian groups show that, for $p$ odd, there are $k+1$ pair-wise permutation non-isomorphic groups in Example~\ref{ex2} when $k>1$ is odd, and $k/2+1$ when $k$ is even. 

The situation when $p=2$ is remarkably different. Indeed, every group in Example~\ref{ex2} is permutation isomorphic to one of the groups in Example~\ref{ex1}.
}
\end{remark}

\begin{proof}[Proof of Theorem~$\ref{thrm:main-2}$]
Let $G$ be a nilpotent transitive group of degree $p^k$ and nilpotency class at most $2$.
Write $Z:=\Zent G$ and let $H$ be the stabiliser in $G$ of a point. As $H$ is core-free in $G$, we get $H\cap Z=1$ and hence $HZ/Z\cong H$. As $G/Z$ is abelian, we obtain that $H$ is abelian and $\langle H,Z\rangle=HZ=H\times Z$. Since $G$ has nilpotency class at most $2$, every subgroup of $G$ containing $Z$ is normal, and hence $HZ\unlhd G$.

Write $p^a:=|Z|$. Thus $|G:HZ|=p^{k-a}$. Let $b$ be the minimum positive integer such that there exist $g_1,\ldots,g_{b}\in G$ with $G=\langle g_1,\ldots,g_{b},HZ\rangle$. Observe that $b\leq k-a$. We claim that $$H\cap \bigcap_{i=1}^{b}H^{g_i}=1.$$ 
Write $Y:=H\cap \bigcap_{i=1}^{b}H^{g_i}$ and take $y\in Y$. For every $i\in \{1,\ldots,b\}$, we have $y\in H^{g_i}$ and hence $y=h_i^{g_i}$, for some $h_i\in H$. Therefore $y=h_1^{g_1}=\cdots =h_{b}^{g_{b}}$. Observe that $h_i^{g_i}=h_i[h_i,g_i]\in H\times Z$ and, since $y\in H$, we get $y=h_i$ and $[h_i,g_i]=1$. In particular,  $y$ is centralised by $g_1,\ldots,g_{b}$. As $y\in H$ and $H$ is abelian, $y$ is centralised by $H$ and hence by $\langle g_1,\ldots,g_b,HZ\rangle=G$. Thus $y\in \Zent G=Z$ and hence $y=1$.

As $HZ\unlhd G$, we see that $H,H^{g_1},\ldots,H^{g_{b}}$ are $b+1$ subgroups of $HZ$ of index $p^a$. As $H\cap \cap_{i=1}^{b}H^{g_i}=1$, we get $|HZ|\leq (p^a)^{b+1}$ from Lemma~\ref{retardd}.  Thus 
\begin{equation}\label{eq0}
|G|=|G:HZ||HZ|\leq p^{k-a}\cdot p^{a(b+1)}\leq p^{k-a+a(k-a+1)}= p^{k+a(k-a)}.
\end{equation}

 The maximum of the function $a\mapsto a(k-a)$ is attained at $a\in \{\lfloor k/2\rfloor,\lceil k/2\rceil\}$. Therefore $|G|\leq p^{k+\lfloor k/2\rfloor\lceil k/2\rceil}$ and hence $F_\textrm{Nil}(p^k,2)\leq p^{k+\lfloor k/2\rfloor\lceil k/2\rceil}$. The groups in Example~\ref{ex1} have order $ p^{k+\lfloor k/2\rfloor\lceil k/2\rceil}$, are transitive of degree $p^k$ and have nilpotency  class at most $2$. Thus $ p^{k+\lfloor k/2\rfloor\lceil k/2\rceil}\leq F_\mathrm{Nilp}(p^k,2)$. This concludes the proof of the first assertion. 

Let $G$ be a transitive group of degree $p^k$, of nilpotency class at most $2$ and with $|G|=F_\mathrm{Nil}(p^k,2)$. From Eq.~\eqref{eq0}, we have $a\in \{\lfloor k/2\rfloor,\lceil k/2\rceil\}$ and $b=k-a$. Hence  the minimum number of generators of $G/HZ$ is $k-a$ and  thus $G/HZ$ is elementary abelian. Again from Eq.~\eqref{eq0} we deduce $|H|=p^{a(k-a)}$. Since $H\cap \cap_{i=1}^{k-a}H^{g_i}=1$, with an inductive argument we  obtain
\begin{equation}\label{label}
|H:H\cap \bigcap_{i\in I}H^{g_i}|=p^{a|I|},
\end{equation}
for every subset $I$ of $\{1,\ldots,k-a\}$.

When $k=1$, $G$ is cyclic of order $p$ and $G$ is one of the groups in Example~\ref{ex2}. Therefore, for the rest of this proof we assume $k>1$. 

For each $i\in \{1,\ldots,k-a\}$, define 
\begin{equation}\label{eq11}C_i:=H\cap \bigcap_{\substack{j=1\\j\neq i}}^{k-a}H^{g_j}.
\end{equation} 
Since $|H|=p^{a(k-a)}$, from Eq.~\eqref{label} we obtain $|C_i|=p^a$. Moreover, as $|H:H\cap H^{g_i}|=p^a$ and $C_i\cap H^{g_i}=1$, we get
\begin{equation}\label{eq12}
H=C_i(H\cap H^{g_i}).
\end{equation}

Since $|HZ:H^{g_i}|=|H:H\cap H^{g_i}|$, we deduce $HZ=HH^{g_i}$ and, as $HH^{g_i}=H[H,g_i]$, we get $HZ=H[H,g_i]$. Now, $[H,g_i]\leq Z$ and hence $$Z=[H,g_i].$$
It follows that the mapping $[-,g_i]:H\to Z$ define by $h\mapsto [h,g_i]$ is surjective. Thus $\cent H{g_i}=\ker([-,g_i])\unlhd H$ and $|H:\cent H{g_i}|=|Z|=p^a$. Since $\cent H {g_i}\leq H\cap H^{g_i}$, we obtain 
\begin{equation}\label{eq01}\cent H {g_i}= H\cap H^{g_i}
\end{equation}
and $\cap_{j=1}^{k-a}\cent H{g_j}=1$. Thus $H$ is a subdirect subgroup of the direct product of $k-a$ copies of $Z$. Since $|H|=|Z|^{k-a}$, we get $H\cong Z^{k-a}$ and $H$ is abelian.

From Eqs.~\eqref{eq11} and~\eqref{eq01}, we get $C_i=\cap_{j\neq i}\cent H {g_j}$. From Eq.~\eqref{eq12}, it follows that $$Z=[H,g_i]=[C_i(H\cap H^{g_i}),g_i]=[C_i\cent H {g_i},g_i]=[C_i,g_i].$$

Since $g_i$ has order $p$ modulo $HZ$ and since $G$ has nilpotency class at most $2$, we have $1=[h,g_i^p]=[h,g_i]^p$ for every $h\in H$, which gives that $[H,g_i]=Z$ has exponent $p$. As $H\cong Z^{k-a}$, $H$ has also exponent $p$.

Next we show that there exist $h_1,\ldots,h_{k-a}\in H$ with $\langle g_1h_1,\ldots,g_{k-a}h_{k-a}\rangle$ abelian. Let $i$ be an element of $\{1,\ldots,k-a\}$. For $j\in \{i+1,\ldots,k-a\}$, let $z_j\in C_j$ with $[g_i,g_j]=[z_j,g_j]^{-1}$ (observe that $[g_i,g_j]\in \gamma_2(G)\leq Z=[C_j,g_j]$ and hence there exists such an element $z_j$). Define $h_i:=\prod_{j=i+1}^{k-a}z_j$. Observe that, for every $r\in\{1,\ldots,k-a\}$, the element $g_r$ commutes with $C_1,\ldots,C_{r-1},C_{r+1},\ldots,C_{k-a}$ and hence $g_r$ commutes with $z_s$, for every $s\in \{1,\ldots,k-a\}$ with $r\neq s$.
Now, let $i,j\in \{1,\ldots,k-a\}$ with $i<j$. Since $G$ has nilpotency class at most $2$ and $H$ is abelian, we have
\begin{eqnarray*}
[g_ih_i,g_jh_j]&=&[g_i,g_j][g_i,h_j][h_i,g_j]=[z_j,g_j]^{-1}\prod_{\ell=j+1}^{k-a}[g_i,z_\ell]\prod_{\ell=i+1}^{k-a}[z_\ell,g_j]\\
&=&[z_j,g_j]^{-1}[z_j,g_j]=1.
\end{eqnarray*}


Write $V:=\langle g_1h_1,\ldots,g_{k-a}h_{k-a},Z\rangle$. Since $Z\leq V$, we have $V\unlhd G$. As $\langle g_1h_1,\ldots,g_{k-a}h_{k-a}\rangle$ is abelian, so is $V$. By construction $V$ is transitive and hence $V\cap H=1$. This gives $|V|=p^k$ and $G$ is isomorphic to the semidirect product $V\rtimes H$. As $G/HZ$ is elementary abelian, so is $V/Z$. Thus $V/Z$ and $Z$ are both elementary abelian. Finally, observe that $H$ centralises every element of $V/Z$ and $Z$. In particular, $G$ is permutation isomorphic to one of the groups in Example~\ref{ex2}.


\end{proof}

\begin{remark}\label{rem2}{\rm Let $N$ denote the number (up to permutation isomorphism) of nilpotent transitive groups of degree $p^k$, nilpotency class $2$ and order $F_\mathrm{Nil}(p^k,2)$. Theorem~\ref{thrm:main-2} and Remark~\ref{rem1} show that
\[
N=
\begin{cases}
0&\textrm{when }k=1,\\
1&\textrm{when }p=2 \textrm{ and }k \textrm{ is even},\\
2&\textrm{when }p=2 \textrm{ and }k>1 \textrm{ is odd},\\
k/2+1&\textrm{when }p>2 \textrm{ and }k \textrm{ is even},\\
k+1&\textrm{when }p>2 \textrm{ and }k>1 \textrm{ is odd}.
\end{cases}
\]}
\end{remark}


\section{Proof of Theorem~$\ref{thrm:main}$}\label{main}

\begin{proof}[Proof of Theorem~$\ref{thrm:main}$]
Let $G$ be a nilpotent transitive permutation group of degree $p^k$ and nilpotency class at most $c$, let $\Omega$ be the set acted upon by $G$, let $\omega\in\Omega$ and let $H:=G_\omega$. 

For $i\in \{1,\ldots,c\}$, define ${a_i}:=\log_p(|\gamma_i(G)H:\gamma_{i+1}(G)H|)$. Observe that $|\gamma_c(G)H:\gamma_{c+1}(G)H|=|\gamma_c(G)H:H|=|\gamma_c(G):\gamma_c(G)\cap H|=|\gamma_c(G)|$ by Lemma~\ref{retard}, and hence $a_c=\log_p(|\gamma_c(G)|)$. Moreover, $k=\log_p(|G:H|)=\sum_{i=1}^c\log_p(|\gamma_i(G)H:\gamma_{i+1}(G)H|)=\sum_{i=1}^ca_i$.

 We show, by induction on $c$, that 
\begin{equation}\label{idiot1}\log_p(|G|)\leq \sum_{i=1}^{c}a_i\left(\frac{\left(\sum_{j=1}^{i-1}a_j\right)^i-1}{\left(\sum_{j=1}^{i-1}a_j\right)-1}\right).
\end{equation}
Denote by $T(a_1,\ldots,a_c)$ the right-hand side of Eq.~\eqref{idiot1}. When $c=1$, the group $G$ is abelian; thus $a_1=k$ and $\log_p(|G|)=k=T(a_1)$. 

Suppose that $c>1$. Let $U$ be the core of $\gamma_c(G)H$ in $G$, that is, $$U:=\bigcap_{g\in G}(\gamma_c(G)H)^g.$$ 
As $\gamma_c(G)\leq U$, we have $\gamma_c(G)H\leq UH$. Moreover, since $U\leq \gamma_c(G)H$, we also have $UH\leq \gamma_c(G)H$. Thus $UH=\gamma_c(G)H$.

We write $\bar{G}:=G/U$ and we adopt the ``bar''notation: for a subgroup $X$ of $G$, $\bar{X}$ denotes $XU/U$.

Let $\mathcal{B}$ be the set of orbits of $\gamma_c(G)$ on $\Omega$. Observe that $\mathcal{B}$ is a $G$-set, that is, the action of $G$ on $\Omega$ induces an action of $G$ on $\mathcal{B}$ by setting $(\delta^{\gamma_c(G)})^g=(\delta^g)^{\gamma_c(G)}$, for every $\delta\in \Omega$ and $g\in G$. 

For $\delta\in \Omega$, we have $|\delta^{\gamma_c(G)}|=|\gamma_c(G):\gamma_c(G)\cap G_\delta|=|\gamma_c(G)|=p^{a_c}$ and hence $|\mathcal{B}|=|\Omega|/p^{a_c}=p^{k-a_c}$. Moreover, the setwise stabiliser of $\omega^{\gamma_c(G)}$ is  $\gamma_c(G)H$ and hence the kernel of the action of $G$ on $\mathcal{B}$ is $U$. Therefore $G/U=\bar{G}$ is a transitive permutation group of degree $p^{k-a_c}$ on $\mathcal{B}$ and the stabiliser in $\bar{G}$ of the point $\omega^{\gamma_c(G)}$ of $\mathcal{B}$ is $\gamma_c(G)H/U=UH/U=\bar{H}$.

Since $U$ fixes setwise every $\gamma_c(G)$-orbit, we have $\omega^{U}\subseteq \omega^{\gamma_c(G)}$. Moreover, as $\gamma_c(G)\leq U$, we get $\omega^{U}=\omega^{\gamma_c(G)}$. From the Frattini argument it follows that $U=\gamma_c(G)(U\cap G_\omega)=\gamma_c(G)(U\cap H)$.

As $\gamma_c(G)\leq U$, we see that $G/U=\bar{G}$ has nilpotency class at most $c-1$. For $i\in \{1,\ldots,c-1\}$, we have 
\[\gamma_i(\bar{G})\bar{H}=\frac{\gamma_i(G)U}{U}\frac{UH}{U}=\frac{\gamma_i(G)\gamma_c(G)(U\cap H)H}{U}=\frac{\gamma_i(G)H}{U}.\]
It follows that 
$$\log_p(|\gamma_i(\bar{G})\bar{H}:
\gamma_{i+1}(\bar{G})\bar{H}|)
=\log_p(|\gamma_i(G)H:\gamma_{i+1}(G)H|)=a_i.$$
 In particular, applying our inductive hypothesis to the permutation group $\bar{G}$ we obtain
\begin{equation}\label{eqqe1}
\log_p(|\bar{G}|)\leq T(a_1,\ldots,a_{c-1}).
\end{equation}

Set $g_0:=1$ and $K_0:=\langle g_0\rangle$. For $i>0$, we define recursively $g_i$ and $K_i$. If $K_{i-1}\gamma_c(G)H$ is a proper subset of $G$, then choose $g_i\in G\setminus K_{i-1}\gamma_c(G)H$ and set $K_i:=\langle K_{i-1},g_i\rangle$. If $G=K_{i-1}\gamma_c(G)H$, then choose $g_i:=1$ and $K_i:=\langle K_{i-1},g_i\rangle=K_{i-1}$. Let $\ell$ be the minimum positive integer with $K_{\ell}=K_{\ell+1}$. By construction, $G=K_\ell \gamma_c(G)H$ and hence $K_\ell$ is transitive on $\mathcal{B}$. Moreover, since $|G:\gamma_c(G)H|=|\mathcal{B}|=p^{k-a_c}$ and $|K_i\gamma_c(G)H|<|K_{i+1}\gamma_c(G)H|$ for each $i\in \{0,\ldots,\ell-1\}$, we have $\ell\leq k-a_c$.

For $\kappa\in \{1,\ldots,c-1\}$, define $$L_{\kappa}:=\bigcap_{i_1,\ldots,i_{c-\kappa}=0}^\ell (\gamma_{\kappa}(G)\cap H\cap U)^{g_{i_1}\cdots g_{i_{c-\kappa}}}.$$
Following verbatim the proof of Proposition~\ref{propo1} and observing that $\gamma_c(G)\leq \Zent G$, we get $L_\kappa=1$. When $\kappa=1$, we deduce that $L_1=1$ is the intersection of $(\ell^{c}-1)/(\ell-1)$ subgroups of $U$ having index $p^{a_c}$. Therefore, as $\ell\leq k-a_c$, by Lemma~\ref{retardd} we get
\begin{equation}\label{qeeq1}
\log_p(|U|)\leq a_c\frac{(k-a_c)^c-1}{(k-a_c)-1}.
\end{equation}

Now, observe that
\begin{eqnarray}\label{lll1}
T(a_1,\ldots,a_c)&=&T(a_1,\ldots,a_{c-1})+a_c\left(\frac{\left(\sum_{j=1}^{c-1}a_j\right)^{c}-1}{\left(\sum_{j=1}^{c-1}a_j\right)-1}\right)\\\nonumber
&=&T(a_1,\ldots,a_{c-1})+a_c\frac{(k-a_c)^c-1}{(k-a_c)-1}.
\end{eqnarray}
From this, Eqs.~\eqref{eqqe1} and~\eqref{qeeq1} yield $\log_p(|G|)=\log_p(|\bar{G}|)+\log_p(|U|)\leq T(a_1,\ldots,a_c)$ and Eq.~\eqref{idiot1} is proved. The definition of $F_\mathrm{Nil}(p^k,c)$ gives $\log_p(F_\mathrm{Nil}(p^k,c))\leq T(a_1,\ldots,a_c)\leq F(k,c)$.
\end{proof}

\begin{proposition}\label{leadcoeff}
Let $p$ be a prime number and let $c$ and $k$ be positive integers.  Then, when $c$ is fixed,
$$F(k,c)=\frac{(c-1)^{c-1}k^c}{c^c}+o(k^{c-1}).$$
\end{proposition}

\begin{proof}

Let $T(a_{1},\ldots,a_{c})$ be $\sum_{i=1}^{c}a_i\left(\frac{\left(\sum_{j=1}^{i-1}a_j\right)^i -1}{\left(\sum_{j=1}^{i-1}a_j\right)-1}\right)$, as in the proof of Theorem~\ref{thrm:main}.   

Consider   $T_{i}(a_{1},\ldots,a_{c}):= a_i\left(\frac{\left(\sum_{j=1}^{i-1}a_j\right)^i-1}{\left(\sum_{j=1}^{i-1}a_j\right)-1}\right)$, the $i$-th summand of $T(a_{1},\ldots,a_{c}) $. We now prove that,
\begin{eqnarray}\label{leadcoeff1} 
\max\left(T_{i}(a_{1},\ldots,a_{c})\mid  a_1,\ldots,a_c\in \mathbb{N}, \sum_{j=1}^{c}a_{j}=k\right) &\leq& \sum_{j=1}^{i}\frac{(j-1)^{j-1} k^{j}}{j^{j}}.
\end{eqnarray}
In fact, since $\sum_{j=1}^{c}a_{j}=k$, we have that  $T_{i}(a_{1},\ldots,a_{c})\leq a_i\left(\frac{\left(k-a_{i}\right)^i-1}{\left(k-a_{i}\right)-1} \right)$ with $0\leq{a_{i}}\leq{k}$.
Observing that,  when $0\leq{a_{i}}\leq{k}$,  the maximum value of  $a_{i}\left(k-a_{i}\right)^{j}$ is  $\frac{j^{j} k^{j+1}}{(j+1)^{j+1}}$ (attained in $a_{i}=\frac{k}{j+1}$), we obtain that  $\frac{\left(k-a_{i}\right)^i-1}{\left(k-a_{i}\right)-1}= \sum_{j=0}^{i-1}a_{i}\left(k-a_{i}\right)^{j} \leq  \sum_{j=1}^{i} \frac{(j-1)^{j-1} k^{j}}{j^{j}}$, and  Eq.~\eqref{leadcoeff1} follows. 

Now, Eq.~\eqref{leadcoeff1} implies that  $$\max\left(T_{i}(a_{1},\ldots,a_{c})\mid  a_1,\ldots,a_c\in \mathbb{N}, \sum_{j=1}^{c}a_{j}=k\right) \leq  \frac{(i-1)^{i-1}k^{i}}{i^{i}}+o(k^{i-1}),$$ from which we conclude  $$F(k,c)
\leq \frac{(c-1)^{c-1} k^{c}}{c^{c}} +o(k^{c-1}).$$

On the other hand,  for every $a_1,\ldots,a_c\in \mathbb{N}$ with $\sum_{j=1}^{c}a_{j}=k$, we have   $T(a_{1},\ldots,a_{c})\geq T_{c}(a_{1},\ldots,a_{c})$. 
As $T_{c}(a_{1},\ldots,a_{c}) = a_c\left(\frac{\left(k-a_{c}\right)^{c}-1}{\left(k-a_{c}\right)-1}\right)$,  it follows that  $
F(k,c)\geq \max\left( a_c\left(\frac{\left(k-a_{c}\right)^{c}-1}{\left(k-a_{c}\right)-1}\right) \mid a_c\in \mathbb{N}, 0\leq a_{c}\leq k \right).$

 Now, for $a_c\in \mathbb{N}$ and  $0\leq a_{c}\leq k $, the maximum of $ a_c\left(\frac{\left(k-a_{c}\right)^{c}-1}{\left(k-a_{c}\right)-1}\right)$   is greater than or equal to 
 $\max \left(  \frac{k}{c}\left(\frac{\left(k-\frac{k}{c}\right)^{c}-1}{\left(k-\frac{k}{c}\right)-1}\right),\left(\frac{k}{c}+1\right)\left(\frac{\left(k-\left(\frac{k}{c}+1\right)\right)^{c}-1}{\left(k-\left(\frac{k}{c}+1\right)\right)-1}\right)  \right).$
A computation shows that  $ \max \left(\frac{k}{c}\left(\frac{\left(k-\frac{k}{c}\right)^{c}-1}{\left(k-\frac{k}{c}\right)-1}\right),\left(\frac{k}{c}+1\right)\left(\frac{\left(k-\left(\frac{k}{c}+1\right)\right)^{c}-1}{\left(k-\left(\frac{k}{c}+1\right)\right)-1}\right)  \right)\geq  \frac{(c-1)^{c-1} k^{c}}{c^{c}} +o(k^{c-1})$. In particular, we conclude  that  $$F(k,c)\geq \frac{(c-1)^{c-1}k^{c} }{c^{c}} +o(k^{c-1})$$ 
and  the proposition follows.  

\end{proof}

Next we show that, for $c$ small compared to $k$, the upper bound in Theorem~\ref{thrm:main} is closed to be optimal. The second author will always be in debt with Laci Kov\'acs for pointing out and discussing~\cite{LS}  in $2005$: in fact, the results in~\cite{LS} have already played a role in the study of transitive $p$-groups, see for example~\cite[Section~$3$]{PS}. 

\begin{proposition}\label{rem3}Let $p$ be a prime number and let $c$ and $k$ be positive integers with $c\mid k$.  Then,
$$\log_p(F_\mathrm{Nil}(p^k,c))\geq \frac{k}{c}{k(c-1)/c\choose c-1}.$$In particular, if $c=o(k)$, then
$$\log_p(F_{\mathrm{Nil}}(p^k,c))\geq \frac{(c-1)^{c-1}}{c^c(c-1)!}k^{c}+o(k^{c-1}).$$ 
\end{proposition}
\begin{proof}
Let $u$ and $v$ be positive integers. Let $U$ be the Galois field $\mathbb{F}_{p^{u}}$ of size $p^{u}$ and let $V$ be an $\mathbb{F}_p$-vector space of dimension $v$ (over the field $\mathbb{F}_p$ of size $p$) with basis $e_1,\ldots,e_{v}$ and dual basis $e_1^*,\ldots,e_{v}^*$.

Let $W:=U\wr V$ be the wreath product of $U$ with $V$. Observe that by considering $U$ and $V$ as transitive regular permutation groups, $W$ is transitive of degree $|U||V|=p^{u+v}$  in its natural imprimitive action. Let $B:=U^V$ be the base group of $W$, that is, $B$ consists of all functions from $V$ to $U$ and $W=B\rtimes V$.

Let $\mathcal{S}$ be the symmetric $U$-algebra on $e_1^*,\ldots,e_{v}^*$ with natural grading $$\mathcal{S}=\bigoplus_{m\in\mathbb{N}}\mathcal{S}_m,$$
where $\mathcal{S}_m$ is spanned over $U$ by the monomials $e_{i_1}^{*j_1}\otimes\cdots \otimes e_{i_\ell}^{*j_\ell}$ of degree 
$m$, that is, $j_1+\cdots +j_\ell=m$. Let $\mathcal{M}_m$ be the $U$-subspace of $\mathcal{S}_m$ spanned  by the monomials 
$e_{i_1}^{*j_1}\otimes\cdots \otimes e_{i_\ell}^{*j_\ell}$ of degree $m$ with $j_r\leq p-1$, for every $r\in \{1,\ldots,\ell\}$. 

Let $\pi:\mathcal{S}\to B$ be the valuation map defined on the basis elements of $\mathcal{S}$ by
$$\pi(e_{i_1}^{*j_1}\otimes \cdots \otimes e_{i_\ell}^{*j_\ell})(x):=(e_{i_1}^*(x))^{j_1}\cdots (e_{i_\ell}^*(x))^{j_\ell}.$$
Set $Z_m:=\pi(\mathcal{M}_m)$. 
 
From~\cite{LS}, we deduce that $W$ has nilpotency class $v(p-1)+1$, $\Zentm m W =\gamma_{v(p-1)+2-m}(W)$ and $B\cap \Zentm m W=\bigoplus_{0\leq i\leq m-1}Z_i$.  Moreover, $\pi_{\vert_{\mathcal{M}_m}}$ is injective and hence $|Z_m|=|\pi(\mathcal{M}_m)|=|\mathcal{M}_m|$.

Now set  $u=k/c$ and $v=k(c-1)/c$ and define $G:=V\ltimes \Zentm c W$. Clearly, $u+v=k$. As $G\leq W$, for $c>0$, $G$ is a nilpotent transitive group of degree $p^k$ and nilpotency class at most $c$. Moreover,
\begin{eqnarray*}
|G|&\geq& |\Zentm c W|\geq |Z_{c-1}|=|\mathcal{M}_{c-1}|\geq p^{u{v\choose c-1}}=p^{\frac{k}{c}{k(c-1)/c\choose c-1}}
\end{eqnarray*}
and the first part of the proposition is proved.

Finally, for $c=o(k)$, we have $$\frac{k}{c}{k(c-1)/c\choose c-1}=\frac{(c-1)^{c-1}}{c^c(c-1)!}k^{c}+o(k^{c-1})$$
and the second part is also proved.
\end{proof}

We conclude tabulating some values for $\log_2(F_\mathrm{Nil}(2^k,c))$ which can be obtain with the computer algebra system \texttt{magma}~\cite{magma}.

\begin{table}[!h]
\begin{tabular}{c|cccccccccccccccc}
$k\backslash c$&1&2&3&4&5&6&7&8&9&10&11&12&13&14&15&16\\\hline
1 & $1$ & $1$    & $1$ &$1$&$1$&$1$&$1$&$1$&$1$&$1$&$1$&$1$&$1$&$1$&$1$&$1$\\
2 & $2$ & $3$    & $3$&$3$&$3$&$3$&$3$&$3$&$3$&$3$&$3$&$3$&$3$&$3$&$3$&$3$\\
3 & $3$ & $5$    &  $6$&$7$&$7$&$7$&$7$&$7$&$7$&$7$&$7$&$7$&$7$&$7$&$7$&$7$\\
4 & $4$ & $8$     &${10}$&${12}$&${13}$&${14}$&${14}$&${15}$&${15}$&${15}$&${15}$&${15}$&${15}$&${15}$&${15}$&${15}$\\
5 & $5$ & ${11}$ &${17}$&$19$&$22$&$25$&$26$&${27}$&${28}$&${29}$&${29}$&${30}$&${30}$&${30}$&${30}$&${31}$\\
\end{tabular}
\caption{Some values for $\log_2(F_\mathrm{Nil}(2^k,c))$}
\end{table}

\end{document}